\documentclass[11pt,a4paper,twoside]{amsart}
\usepackage{amsmath}
\usepackage{amssymb}
\usepackage{amsthm}
\usepackage{mathrsfs}
\usepackage{ascmac}
\usepackage[dvipdfmx]{graphicx}

\allowdisplaybreaks[4]

\makeatletter
 
  \@addtoreset{equation}{section}
 \makeatother

\usepackage{amsrefs}

\def \ds{\displaystyle}
\def \wt{\widetilde}

\def \R{\mathbb{R}}
\def \C{\mathbb{C}}
\def \Z{\mathbb{Z}}

\def \v{\varphi}

\def \e{\varepsilon}
\def \l{\lambda}

\def \d{\delta}

\def \pa{\partial}
\def \n{\nabla}
\def \s{\sigma}
\def \a{\alpha}
\def \b{\beta}
\def \me{\mathcal{E}}
\def \n{\nabla}
\def \t{\theta}

\def \P{\Phi}

\newcommand{\re}{\operatorname{Re}}
\newcommand{\I}{\infty}

\newcommand{\norm}[1]{\left\lVert #1\right\rVert}
\newcommand{\Lebn}[2]{\left\lVert #1 \right\rVert_{L(#2)}}

\newcommand{\Wsn}[2]{\left\lVert #1 \right\rVert_{\dot{W}^{s_2}(#2)}}
\newcommand{\Wsin}[2]{\left\lVert #1 \right\rVert_{\dot{W}^{\s}(#2)}}

\newcommand{\tnorm}[1]{\lVert #1\rVert}

\def\({\left(}
\def\){\right)}
\def\<{\left\langle}
\def\>{\right\rangle}
\def\le{\leqslant}
\def\ge{\geqslant}
\def\ol{\overline}

\DeclareMathOperator{\Lip}{Lip}

\newcommand{\todayd}{\the\year/\the\month/\the\day}
\theoremstyle{plain}
\newtheorem{thm}{Theorem}[section]
\newtheorem{pro}[thm]{Proposition}
\newtheorem{lem}[thm]{Lemma}

\newtheorem{rem}[thm]{Remark}

\theoremstyle{definition}
\newtheorem{dei}[thm]{Definition}

\begin{document}
\title[Global behavior of solutions to GGP equation]{Global behavior of solutions to generalized Gross-Pitaevskii equation}
\author[S. Masaki]{Satoshi MASAKI}
\address[]{Division of Mathematical Science, Department of Systems Innovation, Graduate School of Engineering Science, Osaka University, Toyonaka, Osaka, 560-8531, Japan}
\email{masaki@sigmath.es.osaka-u.ac.jp}
\author[H. Miyazaki]{Hayato MIYAZAKI}
\address[]{Advanced Science Course, Department of Integrated Science and Technology, National Institute of Technology, Tsuyama College, Tsuyama, Okayama, 708-8509, Japan}
\email{miyazaki@tsuyama.kosen-ac.jp}
\date{}

\maketitle

\begin{abstract}
This paper is concerned with time global behavior of solutions to nonlinear Schr\"odinger equation with a non-vanishing condition at the spatial infinity. Under a non-vanishing condition, it would be expected that the behavior is determined by the shape of the nonlinear term around the non-vanishing state.
To observe this phenomenon, we introduce a generalized version of the Gross-Pitaevskii equation, which is a typical equation involving a non-vanishing condition, by modifying the shape of nonlinearity around the non-vanishing state. It turns out that, if the nonlinearity decays fast as a solution approaches to the non-vanishing state, then the equation admits a global solution which scatters to the non-vanishing element for both time directions. 
\end{abstract}

\section{Introduction}
This paper is devoted to the study of nonlinear Schr{\"o}dinger equation 
\begin{align}
\left\{
\begin{array}{l}
	\ds i\pa_t u+\Delta u= \mu \left| |u|^2-1\right|^{p-2}(|u|^2-1)u,\quad (t,x) \in \R^{1+n}, \\
	u(0,x)=u_{0}(x),\quad x\in\R^{n},
\end{array}
\right. \label{nls}
\end{align}
where $n=1,2$, $u(t,x) : \R^{1+n} \rightarrow \C$, $\mu=\pm1$, and $p \geq 2$.
We consider the equation with 
 the ``non-vanishing condition''
\begin{align}
	|u(x)|^{2} \rightarrow 1\quad as\ |x| \rightarrow \infty. \label{ass:1}
\end{align}

Nonlinear Schro\"odinger equations with the non-vanishing condition 
 have been extensively studied in mathematical and physical literatures (\cite{MR1500278}, \cite{PhysRevLett.69.1644}, \cite{MR1696311}, \cite{MR1831831} and references therein). A typical example is Gross-Pitaevskii equation
\begin{align} 
	i\pa_t u+\Delta u=(|u|^2-1)u, \quad |u(x)|^{2} \rightarrow 1\quad (|x| \rightarrow \infty), \label{gp}
\end{align}
which is a model equation for various physical phenomena such as Bose-Einstein condensation (see \cite{10.1063/1.1703944}, \cite{pitaevskii1961vortex}). 
The first approach in the study of the well-posedness for \eqref{gp} appears in Bethuel and Saut \cite{MR1669387}. They proved the global well-posedness in $1+H^1(\R^n)$ for $n=2$, $3$. After that, G{\'e}rard \cite{Ge,Ge2} showed that in energy space $E := \{u \in H^{1}_{loc}(\R^n);\; \n u \in L^2,\; |u|^2-1 \in L^2 \}$. 
A pioneering work on the time global behavior of \eqref{gp} is due to Gustafson, Nakanishi, and Tsai \cite{MR2360438, MR2559713, MR2231117}
(cf. Killip, Oh, Pocovnicu, and Vi{\c{s}}an \cite{MR3039823} for cubic-quintic NLS under \eqref{ass:1}). 

The equation \eqref{nls} is a generalization of the Gross-Pitaevskii equation \eqref{gp}
and is a Hamiltonian evolution associated with a \emph{generalized Ginzburg-Landau energy}
\[
	\me_{p}(u) = \|\n u\|_{L^{2}}^{2} + \frac{\mu}{p} \left\| |u|^2 -1 \right\|_{L^{p}}^{p}.
\]
There are previous attempts to generalize Gross-Pitaevskii equation \eqref{gp} 
by Gallo \cite{MR2424376} and the second author \cite{MR3194504}. 
Remark that their generalization is mainly with respect to a shape of the nonlinearity as $|u|\to 0$ or $|u|\to\I$.

When we consider a nonlinear Schr\"odinger equation 
under \eqref{ass:1},
it seems natural to expect that the behavior of a solution is determined by 
the shape, or more explicitly the decay rate, of the nonlinearity as $|u|\to1$,
not by the shape as $|u|\to0$ or $|u|\to\I$. 
It is worth mentioning that, in all previous works on \eqref{gp} and its generalizations listed above,
 the nonlinearity decays to zero in the rate  $O(|u|^2-1)$
as $|x|\to\I$.
The equation \eqref{nls}
is proposed as a generalization of \eqref{gp} with respect to the decay rate of the nonlinearity 
as $|u|\to1$.
In our equation, the nonlinearity decays in the rate $O((|u|^2-1)^{p-1})$ as $|x|\to\I$.

As a first step of the study of a generalized model of \eqref{gp} in this direction,
we consider the case where the nonlinearity decays faster than the Gross-Pitaevskii model, that is, the case $p>2$. 
We first establish local well-posedness and persistence of regularity results.
Then, the goal is to show that
if $p$ is large then the equation \eqref{nls} admits a solution which 
scatters to the non-vanishing element for both time directions.

To this end, we introduce a transform \eqref{nls} by letting 
$u=1+v$. The equation for $v$ is then
\begin{align}
\left\{
\begin{array}{l}
	\ds i\pa_t v+\Delta v=\mu \left| |v|^{2} + 2\re(v) \right|^{p-2}(|v|^{2} + 2\re(v))(1+v), \\
	v(0,x)=v_{0}(x):=u_0(x)-1.
\end{array}
\right. \label{ggp}
\end{align}
The goal is now to find a scattering solution to this equation.
Here, scattering implies that a solution exists globally in time and asymptotically
behaves like a free solution, $u(t) \sim e^{it\Delta} u_\pm$ ($t\to\pm\I$).
The precise definition is given later.

In this paper, we restrict our attention to the solutions to \eqref{gp} of the form $u=1+v$.
However, the energy space corresponding to $\mathcal{E}_p$ ($p\ge2$) contains other kinds of functions if $n=1,2$.
Indeed, the function $u(x)=\exp ([i\log (1+|x|)]^\alpha)$ ($\alpha<1/2$) has finite energy.
As for the case $p=2$, the structure of the energy space is studied by G\'erard \cite{Ge,Ge2}.
If $n\ge3$ then a function $u$ has a finite $\mathcal{E}_2$ energy is written as $u=e^{i\theta}(1+v)$ with some $\theta\in\R$ and $v\in \dot{H}^1$.
The above example is given in \cite{Ge}.

Denote our new nonlinearity by $F$;
\begin{equation}\label{def:F}
F(v) = \mu \left||v|^{2} + 2\re(v)\right|^{p-2}(|v|^{2} + 2\re(v))(1+v).
\end{equation}
Then, the nonlinearity $F(v)$ satisfies 
\begin{equation}\label{eq:Fv}
|F(v)| \le C\( |v|^{k_1} +|v|^{k_2} \).
\end{equation}
and so
\[
	F(v) = 
	\begin{cases}
	O(|v|^{k_1}) \quad (|v| \to 0), \\
	O(|v|^{k_2}) \quad (|v| \to \I),
	\end{cases}
\] 
where $k_1 = p-1$ and $k_2 = 2p-1$.
As long as we work with a function space with Fatou property\footnote{The property that $|f| \le |g|$ a.e.~implies $\norm{f} \le \norm{g}$.}, such as Lebesgue space,
the nonlinearity $F(v)$ can be handled as a ``gauge variant'' double power type nonlinearity.

As for NLS with the finite sum of power type gauge variant nonlinearities, 
Nakamura-Ozawa \cite{MR1991146} show the small data global existence and scattering
in the case where all the exponent of nonlinearity are larger than or equal to the mass-critical power
\begin{equation}\label{eq:km}
k_{\mathrm{m}}:=1+\frac{4}{n}.
\end{equation} 

Thus, the equation \eqref{ggp} can be handled by the argument in \cite{MR1991146} as long as $p \ge 1+k_{\mathrm{m}}$.
Hence, let us concentrate on the case $p<1+k_{\mathrm{m}}$.
In this case, the lower power $k_1$ in \eqref{eq:Fv} becomes mass-subcritical.
It is known that the scattering problem becomes hard in a mass-subcritical case
because the decay of $L^2$-solution is not sufficient for scattering.
Hence, we need another argument.

There are several methods and techniques to treat mass-subcritical nonlinearities.
A use of the operator $J(t)=x+2it\nabla$ or its fractional power is
well-known tool for study of scattering of solutions to NLS with the mass-subcritical nonlinearity 
(see e.g. \cite{GOV,NO,Ma1,Ma2,KMMV}).
However, The technique heavily relies on the gauge invariant structure of the nonlinearity and so 
it does not seem to be suitable with the analysis of \eqref{ggp}.

Our idea here is to use Kato's argument to handle the mass-subcritical part of the nonlinearity.
Kato \cite{MR1275405} prove the small data global existence and scattering for NLS with the 
gauge variant nonlinearity $f$ satisfying $f \in C^{1}(\C; \C)$, $f(0)=0$ and $f'(z) = O(|z|^{k-1})$ for some $k_{\mathrm{St}} < k < k_{\mathrm{e}}$, where 
\begin{equation}\label{eq:kSt}
k_{\mathrm{St}}:=\frac{n+2+\sqrt{n^2+12n+4}}{2n}
\end{equation}
 and $k_{\mathrm{e}}=1+\frac{4}{n-2}$ ($k_{\mathrm{e}} = \I$ if $n=1,2$).
We would emphasize that the range $k \in (k_{\mathrm{St}},k_{\mathrm{e}})$ includes the mass-subcritical case because $k_{\mathrm{St}} < k_{\mathrm{m}} <k_{\mathrm{e}}$.
The key ingredient is non-admissible Strichartz estimates. 
It is known to be a useful tool to obtain a small data scattering result in mass-subcritical case $k<k_{\mathrm{m}}$
(see, for example, \cite{Ma2,Ma3,MS}). 
Thus, in this paper, we shall consider the case $k_{\mathrm{St}} < k_1  <k_{\mathrm{m}}$ and $k_2<k_{\mathrm{e}}$.
We then encounter the restrictions $n=1,2$ and $1+k_{\mathrm{St}} < p  <1+k_{\mathrm{m}}$.

\subsection{Main results}
To state our main results, we introduce several function spaces. 
Let $k_1=p-1$ and $k_2=2p-1$ be the numbers given in \eqref{eq:Fv}.
Let 
\begin{equation}\label{def:s0}
	s_0 := \max(s_1,s_2), \quad s_1:=\frac{n}2-\frac{n}{k_1+1}, \quad s_2 := \frac{n}{2}-\frac{2}{k_{2} -1}.
\end{equation}
We have $s_0=s_n$ for $n=1,2$ and $p\ge2$.
For an interval $I\subset \R$, we define a function space
\begin{align*} 
X(I) &= L(P_1; I) \cap \dot{W}^{s_2}(P_2; I),
\end{align*}
where $L(P_1; I) = L^{q_1}(I, L^{r_1}(\R^n))$ and $\dot{W}^{s_2}(P_2; I) = L^{q_2}(I, \dot{W}^{s_2, r_2}(\R^n))$ with suitable pairs $(q_j,r_j)$ ($j=1,2$) 
satisfying
\[
	\frac2{q_1} + \frac{n}{r_1} = \frac{2}{k_1-1},\quad
	\frac2{q_2} + \frac{n}{r_2}-s_2 = \frac{2}{k_2-1},
\]
respectively.
For the explicit choice of the exponents, see Section \ref{sec:2}.
We remark that, thanks to the relations, $L(P_1; I)$-norm and $\dot{W}^{s_2}(P_2; I)$-norm are invariant under the scaling
\[
u_{\l}(t,x) = \l^{\frac{2}{k_j -1}} u (\l^2 t, \l x)
\] 
for any $\l >0$ and $j=1,2$, respectively. 
The exponent $s_2$ is a scale critical exponent in such a sense that the scaling
$v_0 (x)  \mapsto \l^{\frac{2}{k_2 -1}}v_{0}(\l x)$ leaves the $\dot{H}^{s_2}$-norm invariant.
We will take $r_1:=p=k_1+1$. The exponent $s_1$ comes from the Sobolev embedding $\dot{H}^{s_1}(\R^d) \hookrightarrow L^{r_1}(\R^d)$.
Denote $f \in X_{\mathrm{loc}}(I)$ if $f \in X(J)$ holds for any compact subinterval $J\subset I$.

Before the scattering problems, we establish existence of solutions.
Throughout this paper, we use the notation $U(t):= e^{it\Delta}$.
\begin{dei}[Solution]\label{def:sol}
We say a function $v(t,x) : I \times \R^d \to \C$ is a solution to \eqref{ggp} on an interval $I \subset \R$, $I \ni 0$
if $v \in X_{\mathrm{loc}}(I)$ and satisfies
\begin{equation}\label{iggp}
	v(t) = U(t)v_0 - i \int_0^t U(t-s) F(v(s)) ds
\end{equation}
in $X_{\mathrm{loc}}(I)$.
We call $I$ is a maximal interval of $v$ if $v(t)$ cannot be
extended to any interval strictly larger than $I$.
We denote the maximal interval of $v$ by $I_{\max} = I_{\max}(v) = (T_{\min}, T_{\max})$.
\end{dei}

We establish local well-posedness results of \eqref{ggp} in the homogeneous Sobolev space.

\begin{thm}[Local well-posedness in $\dot{H}^{s_0} \cap \dot{H}^{s_1}$] \label{thm1:1}
Let $n=1,2$ and $1+k_{\mathrm{St}} < p <1+k_{\mathrm{m}}$.
The Cauchy problem \eqref{ggp} is locally well-posed in $\dot{H}^{s_0} \cap \dot{H}^{s_1}$.
Namely, for any $v_{0} \in\dot{H}^{s_0} \cap \dot{H}^{s_1}$, 
there exists a unique maximal solution $v(t) \in X_{\mathrm{loc}}(I_{\max}) \cap C(I_{\max}, \dot{H}^{s_0} \cap \dot{H}^{s_1})$ to \eqref{ggp} on $I_{\max}$.
Furthermore, the map $v_0 \mapsto v$ is continuous in the following sense:
For any compact $I' \subset I_{\max}$, there exists a neighborhood $V$ of $v_0$ in $\dot{H}^{s_0}$ such that the map 
is Lipschitz continuous from $V$ to $X(I') \cap C(I', \dot{H}^{s_0} \cap \dot{H}^{s_1})$.
\end{thm}

Let us next see if a solution $v(t)$ of \eqref{ggp} belongs to $H^1$ at some time, then $v(t)$ belongs to $H^1$ as long as
it exists and
has the conserved energy
\[
	\me_{p}(1+v(t)) = \|\n v\|_{L^2}^{2} + \frac{\mu}{p}\left\||v|^2 + 2\re(v)\right\|_{L^p}^{p}.
\]
However, it is easy to observe this one, since we establish the following:
\begin{thm}[Persistence of $\dot{H}^\s$-regularity] \label{thm1:2}
Assume $n=1,2$ and $1+k_{\mathrm{St}} < p  <1+k_{\mathrm{m}}$.
Let $v_{0} \in \dot{H}^{s_0} \cap \dot{H}^{s_1}$ and $v (t)$ be a corresponding solution to \eqref{ggp} on $I\subset \R$ given in Theorem
\ref{thm1:1}.
 If $v_0 \in \dot{H}^\sigma $ for some $0\le \sigma < k_1$, then
$v \in C(I,\dot{H}^\s) \cap \dot{W}^\s_{\mathrm{loc}}(P_2;I)$.
\end{thm}

Let us now proceed to our main issue, the scattering problem.
To begin with, we introduce the definition of scattering.
\begin{dei}
We say the solution $v$ scatters forward in time if
$T_{\max}=\I$ and $\lim_{t\to\I} U(-t)v(t)$ exists in some sense.
We say $v$ scatters backward in time if
$T_{\min}=-\I$ and $\lim_{t\to-\I} U(-t)v(t)$ exists in some sense.
\end{dei}
We give a criterion for scattering of the solution to \eqref{ggp},
which is one of the main results of the paper. 
\begin{thm}[Scattering criterion]\label{thm1:5}
Assume $n=1,2$ and $1+k_{\mathrm{St}} < p  <1+k_{\mathrm{m}}$. Let $v_{0} \in \dot{H}^{s_0} \cap \dot{H}^{s_1}$ and $v(t) \in C(I_{\max}, \dot{H}^{s_0} \cap \dot{H}^{s_1}) \cap X_{\mathrm{loc}}(I_{\max})$ be a corresponding maximal solution to \eqref{ggp} on $I\subset \R$ given in Theorem
\ref{thm1:1}.
\begin{enumerate}
\item
If $\norm{v}_{X((0, T_{\max}))} < \I$, then $v(t)$ scatters in $\dot{H}^{s_0} \cap \dot{H}^{s_1}$ for forward time.  
\item
If $v(t)$ scatters forward in time $($in $\dot{H}^{s_0} \cap \dot{H}^{s_1} )$ and if $v(t_0)\in \dot{H}^\s$ for some $t_0 \in I_{\max}$ and $0 \le \sigma <k_1$
then $v(t)$ scatters forward in time in $\dot{H}^\s$
\end{enumerate}
Similar assertions hold for backward in time.
In particular, $v\in X(I_{\max})$ implies $I_{\max}=\R$ and $v$ scatters for both time directions.
\end{thm}

Remark that the criterion is given in terms of solution itself and so that
it is not so easy to check the condition.
We next give two criteria in terms of initial data.
\begin{thm}[Small data scattering I] \label{thm1:3}
Assume $n=1$, $2$ and $1+k_{\mathrm{St}} < p  <1+k_{\mathrm{m}}$. Let $v_{0} \in L^{\frac{n(p-2)}{2}} \cap \dot{H}^{s_2}$. There exists $\delta>0$ such that if 
\[
	\norm{v_0}_{L^{\frac{n(p-2)}{2}}} + \norm{v_0}_{\dot{H}^{s_2}}  \le \delta
\]
then there exists a global solution $v(t) \in X(\R) \cap C(\R, H^{s_2})$ of \eqref{ggp}. Moreover, the solution satisfies 
\[
 \norm{v}_{X(\R)} \le 2\( \norm{v_0}_{L^{\frac{n(p-2)}{2}}} + \norm{v_0}_{\dot{H}^{s_2}} \),
\]
and scatters in $H^{s_2}$ for both time directions.
\end{thm}

\begin{thm}[Small data scattering II]\label{thm1:4}
Assume $n=1,2$ and $1+k_{\mathrm{St}} < p  <1+k_{\mathrm{m}}$. Let $v_0 \in |x|^{-(\frac{2}{p-2}-\frac{n}2)}L^2 \cap \dot{H}^{s_2} $. There exists $\delta>0$ such that if 
\[
	\norm{|x|^{\frac{2}{p-2}-\frac{n}2} v_0}_{L^2} + \norm{v_0}_{\dot{H}^{s_2}}  \le \delta,
\]
then, there exists a global solution $v(t) \in X(\R) \cap C(\R, H^{s_2})$ of \eqref{ggp}. Moreover, the solution satisfies 
\[
 \norm{v}_{X(\R)} \le C \( \norm{|x|^{\frac{2}{p-2}-\frac{n}2} v_0}_{L^2} + \norm{v_0}_{\dot{H}^{s_2}} \),
\]
and scatters in $H^{s_2}$ for both time directions.
\end{thm}

\begin{rem}
We remark that $v_{0} \in L^{\frac{n(p-2)}{2}} \cap \dot{H}^{s_2}$ or 
$v_0 \in |x|^{-(\frac{2}{p-2}-\frac{n}2)}L^2 \cap \dot{H}^{s_2} $
yield $v_0 \in L^2$.
\end{rem}

The rest of this paper is organized as follows. 
In Section \ref{sec:2}, we state non-admissible Strichartz' estimates and collect useful lemmas.
We also introduce several notation and definitions which are used throughout this paper. 
In Section \ref{sec:3}, we turn to the estimates on the nonlinearity of \eqref{ggp}.
Section \ref{sec:4} is devoted to a well-posedenss in a generalized framework.
Finally, we show the main results in Section \ref{sec:5}. 

\section{Preliminary} \label{sec:2}

\subsection{Strichartz' estimates for non-admissible pairs}
To present non-admissible Strichartz' estimates by Kato \cite{MR1275405}, we put the following notations: 
\begin{align*}
&B = (1/2,0), \quad C = (1/2-1/n,1/2) \quad (C=(0, 1/4) \; \text{if $n=1$}),  \\
&D = ((n-2)/2(n-1), n/2(n-1)), \quad (D=(0,1/2) \; \text{if $n=1$}), \\
&E = (1/2-1/n, 1), \quad F = (1/2-1/n, 0) \\
& \hspace{3cm} (E=(0, 1/2),\; F = (0,0) \; \text{if $n=1$}), \\
&B' = (1/2, 1), \quad C' = (1/2+1/n, 1/2) \quad (C'=(1, 3/4) \; \text{if $n=1$}), \\
&E' = (1/2+1/n, 0), \quad F' = (1/2+1/n, 1) \\ 
& \hspace{3cm} (E'=(1, 1/2),\; F' = (1,1) \; \text{if $n=1$}), \\
&T = \triangle (BEF), \quad T' = \triangle (B'E'F'), \quad \hat{T} = \triangle (BCD),
\end{align*}
where the triangles $T$ and $T'$ are open except that $B$ and $B'$ are included, respectively, and $\hat{T} \subset T$, which include the side $]CD[$ (the segment connecting $C$ and $D$ except for edge points) if $n \neq 2$. 
Moreover, we denote $\pi(P) = x+2y/n$ for any $P = (x,y) \in [0,1]^2$. 

We state non-admissible Strichartz' estimates.
 
\begin{pro}[\cite{MR1275405}] \label{stri1:1}
Let $t_0 \in \R$. If $P \in T$, $\ol{P} \in T'$ with $\pi(\ol{P}) - \pi(P) = 2/n$, then it holds that
\[
	\Lebn{\int_{t_{0}}^{t}U(t-s)f(s) ds}{\ol{P}} \le C\Lebn{f}{P}. 
\]  
\end{pro} 
\begin{rem}
Further extension is obtained in \cite{Fo,Vi,Ko}.
However, the above version is sufficient for our purpose.
\end{rem}

\begin{pro}[\cite{MR1275405}] \label{stri1:2}
Let $1/2 < 1/q < m/2(m-1)$ $( 1/2 < 1/q \le 1$ if $n=1 )$. If $P \in \hat{T}$ with $\pi(P) = 1/q$, then it holds that
\[
	\Lebn{U(t)f}{P} \le C\norm{f}_{L^{q}}. 
\]
\end{pro}

Here, we give two useful Lemmas to estimate the nonlinearity. To this end, we introduce a Lipschitz $\mu$ norm $(\mu>0)$. 
For a multi-index $\a=(\a_1,\a_2)\in(\Z_{\ge0})^2$, define $\pa^{\a}=\pa_z^{\a_1} \pa_{\overline{z}}^{\a_2}$.
Put $\mu = N + \b$ with $N \in \Z$ and $\b \in (0,1]$. For a function $G \in C^{N}( \R^2, \C)$, we define
\begin{align*}
\norm{G}_{\Lip \mu} = \sum_{|\a| \le N-1} \sup_{z \in \C \setminus \{0\}} \frac{|\pa^{\a}G(z)|}{|z|^{\mu - |\a|}} + \sum_{|\a| = N}\sup_{z \neq z'} \frac{|\pa^{\a}G(z)-\pa^{\a}G(z')|}{|z-z'|^{\b}}.
\end{align*}
If $G \in C^{N}(\R^2, \C)$ and $\norm{G}_{\Lip \mu} < \I$, then we write $G \in \Lip \mu$.

\begin{lem}[\cite{KPV}] \label{lem1:1}
Assume that $s \ge 0$. Let $p$, $q$, $p_{i}$, $q_{i} \in (1, \I)$ $( i = 1, 2, 3, 4)$. Then, we have 
\begin{align*}
\||D_x|^s(fg)\|_{L^{p}_{x}L^{q}_{t}} \le C(\||D_x|^s f\|_{L^{p_1}_{t}L^{q_1}_{x}}\norm{g}_{L^{p_2}_{t}L^{q_2}_{x}} + \|f\|_{L^{p_3}_{t}L^{q_3}_{x}}\norm{|D_x|^s g}_{L^{p_4}_{t}L^{q_4}_{x}} )
\end{align*}
provided that
\[
\frac{1}{p} = \frac{1}{p_1}+\frac{1}{p_2} = \frac{1}{p_3} + \frac{1}{p_4}, \quad \frac{1}{q} = \frac{1}{q_1}+\frac{1}{q_2} = \frac{1}{q_3} + \frac{1}{q_4},
\]
where the constant $C$ is independent of $f$.
\end{lem}

\begin{lem}[\cite{MR1124294,MS}] \label{lem1:2}
Suppose that $\mu >1$ and $s \in (0,\mu)$. Let $G \in \Lip \mu$. If $p$, $p_1$, $p_2$, $q$, $q_1$, $q_2 \in (1, \infty)$ satisfies 
\[
\frac{1}{p} = \frac{\mu -1}{p_1}+\frac{1}{p_2}, \quad \frac{1}{q} = \frac{\mu -1}{q_1}+\frac{1}{q_2},
\]
then there exists a positive constant $C$ depending on $\mu$, $s$, $p_1$, $p_2$, $q_1$, $q_2$ such that 
\[
\| |D_x|^{s} G(f) \|_{L^{p}_{t}L^{q}_{x}} \le C\|G\|_{\Lip \mu}\|f\|_{L^{p_1}_{t}L^{q_1}_{x}}^{\mu-1}\||D_x|^{s} f\|_{L^{p_2}_{t}L^{q_2}_{x}}
\]
holds for any $f$ satisfying  $f \in L^{p_1}_{t}L^{q_1}_{x}$ and $|D_x|^s f \in L^{p_2}_{t}L^{q_2}_{x}$. 
\end{lem}

Finally, we introduce several notations and definitions which we use throughout this paper.
We define $P_1$, $\ol{P}_1$, $P_2$, $\ol{P}_2$, $P'_2$ and $\ol{P}'_2$ by 
\begin{align*}
&P_1=\( \frac1p, \frac{(2-n)p+2n}{2p(p-2)}\), \quad \ol{P}_1 = \( \frac{p-1}{p}, \frac{(p-1)((2-n)p+2n)}{2p(p-2)}\), \\
&P_2=\( \frac{np^2-2p-2n}{2np(p-1)}, \frac{(2-n)p+2n}{4p(p-1)}\), \\
&\ol{P}_2 = \( \frac{3np^2-2(3n+1)p+2n}{2np(p-1)}, \frac{(2p-1)((2-n)p+2n)}{4p(p-1)}\), \\
&P'_2=\( \frac{p-2}{2p(p-1)}, \frac{(2-n)p+2n}{4p(p-1)}\), \\
&\ol{P}'_2 = \( \frac{(p-2)(2p-1)}{2p(p-1)}, \frac{(2p-1)((2-n)p+2n)}{4p(p-1)}\).
\end{align*}
Remark that the point $P_1$ and $\ol{P}_1$ lye on the line $x+2y/n=2/(n(p-2))$ and
$x+2y/n=2/n+2/(n(p-2))$, respectively.
Similarly, $P'_2$ and $\ol{P}'_2$ are on $x+y=1/(2p-2)$ and
$x+2y/n=1/(n(p-1))$, respectively.
Further, $P_1$, $\ol{P}_1$, $P'_2$, and $\ol{P'}_2$ are on the line $y=(((2-n)p+2n)/2(p-2))x$.
The pair given by $P_2$ and the dual of the pair given by $\ol{P}_2$ are admissible.
Namely, $P_2$ is on $x+2y/n=1/2$ and $\ol{P}_2$ is on $x+y=1/2+2/n$.
One has the relation
\begin{equation}\label{eq:exponents}
	\ol{P}_1 -P_1 = \ol{P}_2 - P_2 
	= \ol{P}'_2 - P'_2 = (k_1-1)P_1  = (k_2-1)P'_2  .
\end{equation}
Unfortunately, note that $P'_2 \not\in T$ and $\ol{P}'_2 \not\in T'$ (see Section \ref{sec:2} for $T$ and $T'$).
We put $s_2 = \frac{n}2 -\frac{1}{p-1}$. Let $k_{\mathrm{St}}=\frac{n+2+\sqrt{n^2+12n+4}}{2n}$ and $k_{\mathrm{m}}=1+\frac{4}{n}$. 
For an interval $I\subset \R$ and a point
$\ds P = \(1/q, 1/r \) \in \square = [0,1]^2$,
$L(P;I)$ denotes $L^{r}(I, L^{q}(\R^2))$.
Similarly, we define $\dot{W}^{s}(P;I) = L^{r}(I, \dot{W}^{s,q}(\R^2))$ 
for $I\subset \R$ and $\ds P = \(1/q, 1/r \) \in \square = [0,1]^2$ and $s \in \R$.
We denote $X(I) = L(P_{1};I) \cap \dot{W}^{s_2}(P_2;I)$. 
If $I=\R$, we omit $\R$ and simply write $L(P)=L(P;\R)$, $\dot{W}^s(P)=\dot{W}^s(P;\R)$, 
and $X=X(\R)$. 

\section{Nonlinear estimates} \label{sec:3}

Let $\v \in C^{\infty}_{0}([0,\infty))$ be a cutoff function satisfying $\v\ge0$,
$\v(s) = 1$ for $s\le 1$ and $\v(s)=0$ for $s\ge2$.
We decompose the nonlinearity $F(z)$ given in \eqref{def:F} as $F(z) = F_1(z) + F_2(z)$, where
\[
F_1(z) = \v(|z|)F(z),\quad F_2(z) = (1-\v(|z|))F(z).
\]
Then, by \eqref{eq:Fv},
\begin{align}
|F_i(z)| \le C|z|^{k_{i}} \quad (i=1,2), \label{nin1:2}
\end{align} 
where $k_1 = p-1$ and $k_2 = 2p-1$. 

Lemma \ref{lem1:1} and Lemma \ref{lem1:2} yields nonlinear estimates as follows:
\begin{lem} \label{lem1:3}
Let $0\le \s < k_1$. The estimates
\begin{align}
& \|F_1(u)\|_{L(\overline{P}_1)} \le C \|u\|_{L(P_{1})} \|u\|_{L(P_1)}^{k_1-1}, \label{non1:0} \\
& \|F_2(u)\|_{L(\overline{P}_1)} \le C \|u\|_{L(P_{1})} \Wsn{u}{P_{2}}^{k_2-1}, \label{non1:1} \\
& \norm{F_1(u)}_{\dot{W}^\s(\overline{P}_2)} \le C \norm{u}_{\dot{W}^\s(P_{2})}\|u\|_{L(P_{1})}^{k_1-1}, \label{non1:2} \\
& \norm{F_2(u)}_{\dot{W}^\s(\overline{P}_2)} \le C \norm{u}_{\dot{W}^\s(P_{2})}\norm{u}_{\dot{W}^{s_2}(P_{2})}^{k_2-1} \label{non1:3} 
\end{align}
hold for any $u \in X \cap \dot{W}^\s (P_2)$. 
\end{lem}

\begin{lem} \label{lem1:4}
Let $0\le \s < k_1 -1 $. The following four estimates hold
for any $u$, $\wt{u} \in X \cap \dot{W}^\s (P_2):$
\begin{equation}\label{non2:1}
	\Lebn{F_1(u)-F_1(\wt{u})}{\ol{P}_1} 
	\le C\Lebn{u-\wt{u}}{P_1}\( \Lebn{u}{P_1}^{k_1-1} + \Lebn{\wt{u}}{P_1}^{k_1-1} \),
\end{equation} 
\begin{equation}\label{non2:2}
	\Lebn{F_2(u)-F_2(\wt{u})}{\ol{P}_1} 
	\le C\Lebn{u-\wt{u}}{P_1}\( \Wsn{u}{P_2}^{k_2-1} + \Wsn{\wt{u}}{P_2}^{k_2-1} \),
\end{equation} 
\begin{multline}\label{non2:3}
	\norm{F_1(u)-F_1(\wt{u})}_{\dot{W}^{\s}(\ol{P}_2)}  \\
	\le C\Lebn{u-\wt{u}}{P_1}\( \Lebn{u}{P_1}^{k_1-2} + \Lebn{\wt{u}}{P_1}^{k_1-2} \)\( \Wsin{u}{P_2} + \Wsin{\wt{u}}{P_2} \) \\
	C\norm{u-\wt{u}}_{\dot{W}^\s(P_2)}\( \Lebn{u}{P_1}^{k_1-1} + \Lebn{\wt{u}}{P_1}^{k_1-1} \),
\end{multline} 
and
\begin{multline} \label{non2:4}
	\norm{F_2(u)-F_2(\wt{u})}_{\dot{W}^{\s}(\ol{P}_2)}  \\
	\le{} C\norm{u-\wt{u}}_{\dot{W}^{s_2}(P_2)}\( \Wsn{u}{P_2}^{k_2-2} + \Wsn{\wt{u}}{P_2}^{k_2-2} \)\( \Wsin{u}{P_2} + \Wsin{\wt{u}}{P_2} \) \\
	{}+C\Wsin{u-\wt{u}}{P_2}\( \Wsn{u}{P_2}^{k_2-1} + \Wsn{\wt{u}}{P_2}^{k_2-1} \) .
\end{multline} 
\end{lem}

\begin{proof}[Proof of Lemma \ref{lem1:3}]
First, we show \eqref{non1:1}. Since
$\overline{P_1} = P_1+(k_2-1)P'_2$ by \eqref{eq:exponents},
it follows from $\dot{W}^{s_2}(P_{2}) \hookrightarrow L(P'_2)$ that 
\begin{align*}
	\Lebn{F_2(u)}{\ol{P_1}} &\le \Lebn{u}{P_1}\Lebn{u}{P'_2}^{k_2-1} \\
	&\le C\Lebn{u}{P_1}\Wsn{u}{P_2}^{k_2-1}.
\end{align*}
The estimate \eqref{non1:0} follows in a similar way.

Let us next show \eqref{non1:2}.
Since we have
$\overline{P_2} = P_2+(k_1-1)P_1$ by \eqref{eq:exponents}, 
it follows from Lemma \ref{lem1:2} that
\begin{align*}
	\norm{F_1(u)}_{\dot{W}^\s({\overline{P}_2})} &\le C\norm{F_1}_{\Lip(k_1)}\norm{u}_{\dot{W}^\s(P_{2})} \|u\|_{L(P_{1})}^{k_1-1} \\
	&\le C\norm{u}_{\dot{W}^\s(P_{2})} \|u\|_{L(P_{1})}^{k_1-1},
\end{align*}
which implies \eqref{non1:2}.
The last inequality \eqref{non1:3} follows in a similar way.
\end{proof}

\begin{proof}[Proof of Lemma \ref{lem1:4}]
The proofs of \eqref{non2:1} and \eqref{non2:2} are similar to that of \eqref{non1:0} an \eqref{non1:1},
respectively. We omit the details.
Let us prove \eqref{non2:4}.
Note that 
\begin{align*}
	F_i(u)-F_i(\wt{u}) ={}& (u-\wt{u}) \int_{0}^{1} \pa_{z}F_1(\t u + (1-\t)\wt{u}) d\t \\
	&+ \overline{(u-\wt{u})} \int_{0}^{1} \pa_{\bar{z}}F_1(\t u + (1-\t)\wt{u}) d\t \\
	=:{}& I_{i} +J_{i}
\end{align*}
for $i=1$, $2$. 
Using the relation
$\ol{P}_2 = P_2' + (k_2-2)P'_2 + P_2 = P_2 + (k_1-1)P_2'$,
Lemma \ref{lem1:1}, and Lemma \ref{lem1:2}, one sees that
\begin{align*}
	&\Wsin{I_2}{\ol{P}_2} \\
	\le{}& C\Lebn{u-\wt{u}}{P_2'}\int_{0}^{1} \Lebn{|D_x|^\s \{ \pa_{z}F_2(\t u + (1-\t)\wt{u}) \} }{(k_2-2)P_2' +P_2} d\t \\
	&+ C\Lebn{|D_x|^\s (u-\wt{u})}{P_2}\int_{0}^{1} \Lebn{\pa_{z}F_2(\t u + (1-\t)\wt{u})}{(k_2-1)P_2'} d\t \\
	\le{}& C\Lebn{u-\wt{u}}{P_2'}\int_{0}^{1} \norm{\pa_z F_2}_{\Lip(k_2-1)} \\ 
	& \times \Lebn{ \t u + (1-\t)\wt{u} }{P_2'}^{k_2-2} \Lebn{|D_x|^{\s} (\t u + (1-\t)\wt{u}) }{P_2} d\t \\
	& + C\Lebn{|D_x|^\s (u-\wt{u})}{P_2}\int_{0}^{1} \Lebn{\t u + (1-\t)\wt{u}}{P_2'}^{k_2-1} d\t \\
	\le{}& C\Lebn{u-\wt{u}}{P_2'}\( \Lebn{u}{P_2'}^{k_2-2} + \Lebn{\wt{u}}{P_2'}^{k_2-2} \)\( \Wsin{u}{P_2} + \Wsin{\wt{u}}{P_2} \) \\
	& +C\Wsin{u-\wt{u}}{P_2}\( \Lebn{u}{P_2'}^{k_2-1} + \Lebn{\wt{u}}{P_2'}^{k_2-1} \) .
\end{align*}
Hence this term is handled by the Sobolev embedding.
In the same way, it holds that 
\begin{align*}
	\Wsin{J_2}{\ol{P}_2} \le{}& C\norm{u-\wt{u}}_{\dot{W}^{s_2}(P_2)}\( \Wsn{u}{P_2}^{k_2-2} + \Wsn{\wt{u}}{P_2}^{k_2-2} \) \\ 
	&{}\times \( \Wsin{u}{P_2} + \Wsin{\wt{u}}{P_2} \) \\
	&{}+C\Wsin{u-\wt{u}}{P_2}\( \Wsn{u}{P_2}^{k_2-1} + \Wsn{\wt{u}}{P_2}^{k_2-1} \).
\end{align*}
These yield \eqref{non2:4}. Similarly, we can show \eqref{non2:3}. We only remark that we use the relation 
$\ol{P}_2 = P_1 + (k_1-2)P_1 + P_2 = P_2 + (k_1-1)P_1$ instead.
\end{proof}

\section{Local well-posedness}\label{sec:4}
In this section, we establish a weak version of local well-posedness type result for the equation
\begin{equation}\label{eq:aeq}
	v(t) = V(t) -i \int_{t_0}^t U(t-s) F(v(s))ds,
\end{equation}
where $t_0 \in \R$ is the initial time and $V(t) \in X_{\mathrm{loc}}(\R)$ is a given function.
We call $V(t)$ as a \emph{guide flow}.
\begin{dei}
We say a function $v(t)$ is a \emph{solution to \eqref{eq:aeq} associated with an initial time $t_0$ and a guide flow $V(t)$ on an interval $I$} if 
$t_0$ is in the closure of $I$, $v\in X_{\mathrm{loc}}(I)$, and \eqref{eq:aeq} holds in $X_{\mathrm{loc}}(I)$ sense.
\end{dei}
\begin{rem}\label{rmk:guideflow}
A typical example of guide flow is a linear flow $U(t-t_0)v_0$.
Then, \eqref{eq:aeq} becomes \eqref{iggp}.
Another example is $V(t)=U(t-t_0)v_0 - i\int_{t_0}^t U(t-s)e(s) ds$ for some function $e(t)$.
Then, the equation \eqref{eq:aeq} corresponds to \eqref{ggp} with error; $i\partial_t v+ \Delta v = F(v) +e$.
By using the above formulation, we can handle these examples in a unified way.
\end{rem}
\begin{rem}
Any solution to \eqref{eq:aeq} associated with $t_0$ and $V$ on $I$ satisfies
$v-V \in C(I,\dot{H}^{s_2})$. Indeed, for any compact $J \subset I$, we have $v\in X(J)$.
By means of \eqref{non1:2} and \eqref{non1:3}, this implies $F(v) \in \dot{W}^{s_2}(\ol{P}_2;J)$.
Then, Strichartz estimate shows
\[
	\norm{v-V}_{L^\I(J,\dot{H}^{s_2})} = 
	\norm{\int_{t_0}^t U(t-s)F(u(s))ds}_{L^\I(J,\dot{H}^{s_2})} \le C \norm{F(v)}_{\dot{W}^{s_2}(\ol{P}_2;J)}.
\]
Hence, $v-V \in C(J,\dot{H}^{s_2})$.
\end{rem}
The heart of the analysis of \eqref{eq:aeq} is summarized as follows.
\begin{pro}\label{prop1:2}
Let $\wt{v}(t)$ be a solution to \eqref{eq:aeq} associated with $t_0\in \R$ and a guide flow $\wt{V}(t) \in X_{\mathrm{loc}}(\R)$ on an interval $I\ni t_0$. 
Suppose that $\norm{\wt{v}}_{X(I)} \le M$.
Then, there exists $\eta_1=\eta_1(M)>0$ such that if a guide flow $V(t) \in X_{\mathrm{loc}}(\R)$ satisfies
\[
	\eta:= \tnorm{V-\wt{V}}_{X(I)} \le \eta_1
\]
then there exists a unique solution ${v}(t)$ to \eqref{eq:aeq} associated with $t_0$ and $V(t)$ on the same interval $I$.
Further, the solution satisfies $v-\wt{v}-(V-\wt{V})\in C(I, \dot{H}^{s_2})$ and
\[
	\norm{v-\wt{v}}_{X(I)} + \tnorm{v-\wt{v}-(V-\wt{V})}_{L^\I(I,\dot{H}^{s_2})} \le C \eta.
\]
\end{pro}

\begin{proof}
We prove the result with replacing $I$ with $ I \cap \{t > t_0\} $ since the other case $ I \cap \{t < t_0\} $
is handled in the same way.
Hence, we may suppose that $t_0 = \inf I$.

Let $m>0$ be a small number to be chosen later.
Then, there exists a subdivision $\{I_j\}_{j=1}^J$ of $I$ such that $J=J(m,M)\ge1$, $I_j = (t_{j-1},t_j)$ for $j\le J-1$,
$I_J=(t_{J-1},\sup I)$, and that $\sup_j \norm{\wt{v}}_{X(I_j)} \le m$.
Set $w(t):={v}(t)-\wt{v}(t)$ and $W(t):=V(t)-\wt{V}(t)$. Then, $w(t)$ solves
\begin{equation}\label{eq:lts_pf1}
	w(t) = W(t) -i \int^{t}_{t_0} U(t-s) (F(w(s)+\wt{v}(s))-F(\wt{v}(s))) ds,
\end{equation}
at least formally. We now regard \eqref{eq:lts_pf1} as an equation with respect to $w$.
  
Let us show that there exists a unique function $w\in X(I)$ satisfying \eqref{eq:lts_pf1} in $X(I)$.
We use an induction argument. 
Introduce a map
\[
\P(w)(t) := W(t) -i \int^{t}_{t_0} U(t-s) (F(w(s)+\wt{v}(s))-F(\wt{v}(s))) ds.
\]
Let $\{a_j\}_{1\le j \le J}$ and $\{b_j\}_{1\le j \le J}$ be two sequences to be determined later.
Define
\begin{align*}
X_{1}:=&\{u \in X(I_1) \ |\ \norm{u}_{X(I_1)} \le a_1 \},& d_{X_1}(u, \wt{u} ) :=& \norm{u-\wt{u}}_{L(P_1;I_1)}
\end{align*}
and
\begin{align*}
	&X_j := \left\{ u \in X(\cup_{k=1}^j I_k) \middle| 
	\begin{aligned}
	& u(t) = w(t) \text{ on } \cup_{k=1}^{j-1} I_k, \\
	&\tnorm{F(u+\wt{v})-F(\wt{v})}_{L(\ol{P}_1;\cup_{k=1}^{j-1} I_k) \cap \dot{W}^{s_2}(\ol{P}_2;\cup_{k=1}^{j-1} I_k)} \le b_{j}, \\
	& \norm{u}_{X(\cup_{k=1}^j I_k) } \le a_j.
	\end{aligned}
	\right\}, \\
	&d_{X_j}(u, \wt{u} ) := \norm{u-\wt{u}}_{L(P_1; \cup_{k=1}^j I_k)}.
\end{align*}
for $2\le j \le J$, where $w(t)$ is a solution to \eqref{eq:lts_pf1} on $\cup_{k=1}^{j-1} I_k$.
Our strategy is as follows.
We first construct a solution $w(t)$ on $I_1$ by applying the contraction mapping principle in $X_1$.
Once a solution $w(t)$ is given on $\cup_{k=1}^{j-1}I_k$ for some $j\ge2$, the space $X_j$ is well-defined.
Then, we extend $w(t)$ to the interval $\cup_{k=1}^{j}I_k$
by proving that $\P:\ X_{j} \rightarrow X_{j}$ is a contraction map.

We consider $j\ge2$. Assume that a solution $w(t)$ exists on $\cup_{k=1}^{j-1}I_k$ and that the solution satisfies
$\tnorm{F(w+\wt{v})-F(\wt{v})}_{L(\ol{P}_1;\cup_{k=1}^{j-1} I_k) \cap \dot{W}^{s_2}(\ol{P}_2;\cup_{k=1}^{j-1} I_k)} \le b_{j}$.
Let $u \in X_j$. We have $u=w(t)$ on $t\in \cup_{k=1}^{j-1}I_k$ and so
\begin{equation}\label{eq:aeq_pf1}
	\P(u)(t) = W(t) -i \int^{t}_{t_0} U(t-s) (F(w(s)+\wt{v}(s))-F(\wt{v}(s))) ds = w(t)
\end{equation}
for $t\in \cup_{k=1}^{j-1}I_k$. Then, together with the assumption on $w(t)$, we see that
\begin{equation}\label{eq:aeq_pf2}
	\tnorm{F(\P(u)+\wt{v})-F(\wt{v})}_{L(\ol{P}_1;\cup_{k=1}^{j-1} I_k) \cap \dot{W}^{s_2}(\ol{P}_2;\cup_{k=1}^{j-1} I_k)} \le b_{j}
\end{equation}
Further,
it follows from Strichartz' estimate, the assumption on $W$, and the definition of $X_j$ that
\begin{equation}\label{eq:aeq_pf3}
\begin{aligned}
	\norm{\P(u)}_{X(\cup_{k=1}^{j} I_k)} \le{}& \norm{W}_{X(\cup_{k=1}^{j} I_k)}\\
	&{}+ C \tnorm{F(u+\wt{v})-F(\wt{v})}_{L(\ol{P}_1;\cup_{k=1}^{j} I_k) \cap \dot{W}^{s_2}(\ol{P}_2;\cup_{k=1}^{j} I_k)} \\
	\le{}& \eta + C_1 b_j + 
	C_1\tnorm{F(u+\wt{v})-F(\wt{v})}_{L(\ol{P}_1;I_j) \cap \dot{W}^{s_2}(\ol{P}_2;I_j)}.
\end{aligned}
\end{equation}
Without loss of generality, we may suppose that $C_1 \ge 1$.
Using Lemma \ref{lem1:4} and Young's inequality and letting $m$ small enough, we obtain
\begin{align*}
	&C_1\tnorm{F(u+\wt{v})-F(\wt{v})}_{L(\ol{P}_1;I_j) \cap \dot{W}^{s_2}(\ol{P}_2;I_j)}\\
	&\le C \norm{u}_{X(I_j)} (\norm{u}_{X(I_j)}^{k_1-1}+\norm{\wt{v}}_{X(I_j)}^{k_1-1}+ \norm{u}_{X(I_j)}^{k_2-1}+\norm{\wt{v}}_{X(I_j)}^{k_2-1}  ) \\
	&{}\le \frac14 \norm{u}_{X(I_j)} + C_2(m) \norm{u}_{X(I_j)}^{k_2}.
\end{align*}
Fix such $m$. Then, as long as $a_j \le (4C_2)^{-1/(k_2-1)}$,
\begin{equation}\label{eq:aeq_pf4}
 C_1\tnorm{F(u+\wt{v})-F(\wt{v})}_{L(\ol{P}_1;I_j) \cap \dot{W}^{s_2}(\ol{P}_2;I_j)}
 \le \frac12 \norm{u}_{X(I_j)} \le \frac{1}{2} a_j.
\end{equation}
Combining \eqref{eq:aeq_pf1}--\eqref{eq:aeq_pf4}, we show that $\P: X_j \to X_j$ if 
\begin{equation}\label{eq:aeq_pf5}
	2(\eta + C_1 b_j) \le a_j \le (4C_2)^{1/(k_2-1)}.
\end{equation}
Remark that the condition works also for $j=1$ with the choice $b_1=0$. 

We will show that $\P$ is a contraction map. By \eqref{eq:aeq_pf1},
\[
	\P(u_1)-\P(u_2) = -i\int_{t_{j-1}}^t U(t-s) (F(u_1+\wt{v})-F(u_2+\wt{v})) ds
\]
for $u_1,u_2 \in X_j$.
A use of \eqref{non2:1} and \eqref{non2:2} then shows
\[
	d_{X_j}(\P(u_1),\P(u_2)) \le C_3\norm{u_1-u_2}_{L(P_1;I_j)} (a_j^{k_1-1} + a_j^{k_2-1} + m^{k_1-1}+ m^{k_2-1})
\]
We let $m$ even small so that $C_3(m^{k_1-1}+ m^{k_2-1}) \le 1/3$, if necessary.
Then, $\P: X_j\to X_j$ is contraction if 
\begin{equation}\label{eq:aeq_pf6}
	a_j \le \min (1, (6C_3)^{-\frac{1}{k_1-1}}).
\end{equation}

Thus, if \eqref{eq:aeq_pf5} and \eqref{eq:aeq_pf6} are satisfied then $\P:X_j\to X_j$ is a contraction and so 
we obtain a solution $w(t) \in X_j \subset X(\cup_{k=1}^{j}I_k)$ to \eqref{eq:lts_pf1} on $\cup_{k=1}^j I_{k}$.
For the next step of the induction, we shall define $b_{j+1}$.
By the assumption of the induction, $C_1\ge1$, and \eqref{eq:aeq_pf4}, we have
\[
	\norm{F(w+\wt{v})-F(\wt{v})}_{X(\cup_{k=1}^j I_k)} 
	\le b_j + \frac12 a_j.
\]
Hence, it suffices to take
\begin{equation}\label{eq:aeq_pf7}
	b_j + \frac12 a_j \le b_{j+1}.
\end{equation}

Now, the proof is completed if we are able to choose two sequences $\{a_j\}_{j}$ and $\{b_j\}_{j}$ so that
\eqref{eq:aeq_pf5}, \eqref{eq:aeq_pf6}, and \eqref{eq:aeq_pf7} are satisfied.
Recall that $b_1=0$.
We take $a_j=2(\eta+C_1 b_j)$ and $b_{j+1}+\eta = (C_1+1)(b_j+\eta)$, or more explicitly,
\[
	a_j = 2(C_1(C_1+1)^{j-1}-C_1+1) \eta ,\quad b_j = ((C_1+1)^{j-1}-1) \eta .
\]
Notice that $a_j$ is increasing. Thus, one sees that if $\eta>0$ is taken so small that
\[
	 \eta \le \frac{\min ((4C_2)^{1/(k_2-1)},1,(6C_3)^{-\frac{1}{k_1-1}})}{2(C_1(C_1+1)^{J-1}-C_1+1)} = : \eta_1
\]
then the conditions \eqref{eq:aeq_pf5}, \eqref{eq:aeq_pf6}, and \eqref{eq:aeq_pf7} are satisfied for all $j\in[1,J]$.
Thus, there exists a unique solution $v(t)$ to \eqref{eq:aeq} associated with $t_0$ and $V(t)$ on $I$.
Furthermore the solution satisfies
$\norm{v-\wt{v}}_{X(I)} \le C\eta$.

Let us finally estimate $w-W$. By using \eqref{eq:lts_pf1} and the Strichartz estimate,
\begin{align*}
	&\norm{w-W}_{L^\I(I, \dot{H}^{s_2})}\\
	&\le C \norm{F(w+\wt{v})-F(\wt{v})}_{ \dot{W}^{s_2}(\ol{P}_2;I)}.\\
	&\le C \norm{w}_{X(I)} (\norm{w}_{X(I)}^{k_1-1} + \norm{w}_{X(I)}^{k_2-1}+ \norm{\wt{v}}_{X(I)}^{k_1-1} + \norm{\wt{v}}_{X(I)}^{k_2-1}) 
	\le C \eta,
\end{align*}
which completes the proof.
\end{proof}

\begin{thm}[Local existence of a solution]\label{thm:aeq_lwp1}
There exists  a universal constant $\delta_0>0$ such that if
a guide flow $V(t)$ satisfies $\norm{V}_{X(I)} \le \delta_0$
for some interval $I$ then for any $t_0 \in \R \cap \ol{I}$ there exists
a unique solution $v(t)$ of \eqref{eq:aeq} associated with $t_0$ and $V(t)$ on $I$.
Moreover, $v\in X(I)$.
\end{thm}
\begin{proof}
Take $\wt{V}\equiv0$ and $\wt{v} \equiv 0$ in Proposition \ref{prop1:2}.
\end{proof}
\begin{thm}[Uniqueness and unique continuation]\label{thm:aeq_lwp2}
Let $v_1,v_2$ be two solutions of \eqref{eq:aeq} associated with $t_0$ and $V(t)$ on intervals $I_1$ and $I_2$, respectively.
If $v_j \in X(I_j)$ for $j=1,2$ and if $t_0 \in I_1 \cap I_2$ then $v_1=v_2$ on $I_1 \cap I_2$.
In particular, under the same assumption, both solutions can be uniquely extended to a solution on $I_1 \cup I_2$.
\end{thm}
\begin{proof}
We apply $\wt{V}\equiv V$, $\wt{v} =v_1$, $v=v_2$, and $I=I_1\cap I_2$.
Then, we obtain $v_2=v_1$ in $X(I)$ on $I_1\cap I_2$.
Unique continuation property is obvious.
\end{proof}

Let $v(t)$ be a solution of \eqref{eq:aeq} associated with $t_0$ and $V(t)$ on $I$.
We define $I_{\max}=I_{\max}(t_0,V)=(T_{\min}(t_0,V),T_{\max}(t_0,V))$, where
\begin{align*}
	T_{\max} &:= \sup \left\{ T>t_0 \ \middle|\ 
	\begin{aligned}
	&\exists\text{a solution }v
	\text{ associated with }
	t_0
	\text{ and }
	V(t)
	\\
	&	\text{ on }
	[t_0,T]\text{ satisfying }X([t_0,T])<\I.
	\end{aligned}
	\right\}, \\
	T_{\min} &:= \inf \left\{ T<t_0 \ \middle|\ 
	\begin{aligned}
	&\exists\text{a solution }v
	\text{ associated with }
	t_0
	\text{ and }
	V(t)
	\\
	&	\text{ on }
	[t_0,T]\text{ satisfying }X([T,t_0])<\I.
	\end{aligned}
	\right\}.
\end{align*}
By Theorems \ref{thm:aeq_lwp1} and \ref{thm:aeq_lwp2}, for any $t_0$ and $V(t)$ there exists
a unique solution $v(t)$ associated with $t_0$ on $V(t)$ on $I_{\max}(t_0,V)\ni t_0$.
Remark that $v(t)\in X_{\mathrm{loc}}(I_{\max}(t_0,V))$.
We call this solution a \emph{maximal solution}.

To complete our well-posedness type result on \eqref{eq:aeq}, 
we shall observe continuous dependence of guide flow.

\begin{thm}[Continuous dependence on guide flow]\label{thm:aeq_lwp3}
Fix $t_0\in\R$. The mapping from a guide flow $V(t) \in X_{\mathrm{loc}}(\R)$ to
a maximal solution $v(t)\in X_{\mathrm{loc}}(I_{\max}(t_0,V))$ is continuous in the following sense:
For any compact interval $J \subset I_{\max}(t_0,V)$ and a positive number $\e>0$,
there exists a neighborhood $\mathcal{V}=\mathcal{V}(t_0,V,J,\e)\subset X(J)$
of $V(t)$ such that if $\wt{V} \in \mathcal{V}$ then $I_{\max}(t_0, \wt{V}) \supset J$ and
a maximal solution $\wt{v}$ of \eqref{eq:aeq} associated with $t_0$ and $\wt{V}$ satisfies
$\norm{v-\wt{v}}_{X(J)} \le \e$.
\end{thm}
\begin{proof}
This is merely a qualitative version of Proposition \ref{prop1:2}.
\end{proof}

\begin{rem}
If we take both $\wt{V}$ and $V$ as in the second example of Remark \ref{rmk:guideflow},
then Proposition \ref{prop1:2} reads as a stability result, which is sometimes called a long-time stability.
\end{rem}

We conclude this section with a regularity property.
\begin{thm}[inheritance of regularity]\label{thm:regularity}
Let $t_0\in\R$ and let $V(t)$ be a guide flow.
Let $v(t)$ be a unique maximal solution to \eqref{eq:aeq} associated with $t_0$ and $V$ on $I_{\max}$.
If $V(t) \in C(I_{\max}, \dot{H}^\sigma) \cap \dot{W}^{\s}_{\mathrm{loc}}({P}_2;I_{\max})$ for some $0\le \sigma < k_1$
then $v(t)\in C(I_{\max}, \dot{H}^\sigma) \cap \dot{W}^{\s}_{\mathrm{loc}}({P}_2;I_{\max})$.
Further, the equation \eqref{eq:aeq} holds in $\dot{H}^{\s}$ sense for all $t\in I_{\max}$.
\end{thm}
\begin{proof}
Let $I'$ be a compact interval such that $I' \subset I_{\max}$.
We claim that $v \in \dot{W}^{\sigma}(P_2;I')$.
Remark that the claim shows the result because we deduce from Strichartz' estimate,
\eqref{non1:2}, and \eqref{non1:3} that
\begin{align*}
&\norm{\int_{t_0}^t U(t-s) F(v(s)) ds }_{L^\I(I',\dot{H}^\sigma) \cap \dot{W}^\sigma (P_2;I')} \\
	&{}\le C \norm{F(v)}_{\dot{W}^\s (\ol{P}_2;I')}\\
	&{}\le C \norm{v}_{\dot{W}^{\s}(P_2;I')}(\norm{v}_{X(I')}^{k_1-1} + \norm{v}_{X(I')}^{k_2-1}) <\I.
\end{align*}
Together with $V(t) \in C(I_{\max}, \dot{H}^\sigma) \cap \dot{W}^{\s}_{\mathrm{loc}}({P}_2; I_{\max})$,
it proves the desired result.

We show the claim.
As in the proof of Proposition \ref{prop1:2}, it suffices to show under the assumption $\inf I' = t_0$.
Fix $m>0$.
Divide $I'$ into $J(m)$ intervals $I_j:=[t_{j-1},t_j]$ by choosing suitable
$t_0 < t_1 < t_2 < \dots < t_{J} = \sup I'$ so that $\max_{j} \norm{v}_{X(I_j)} \le m$.
Let us show $v \in \dot{W}^{\sigma}(P_2; \cup_{k=1}^j I_k)$ for all $j$ by induction on $j$.
To this end, we take $j\ge2$ and suppose that $v \in \dot{W}^{\sigma}(P_2;\cup_{k=1}^{j-1} I_k)$.
Then, $F(v) \in \dot{W}^\s (\ol{P}_2;\cup_{k=1}^{j-1} I_k)$ in light of \eqref{non1:2} and \eqref{non1:3}.
By Strichartz' estimate and \eqref{non1:2} and \eqref{non1:3},
\begin{align*}
&\norm{\int_{t_0}^t U(t-s) F(v(s)) ds }_{\dot{W}^\sigma (P_2;\cup_{k=1}^j I_k)} \\
	&{}\le C \norm{F(v)}_{\dot{W}^\s (\ol{P}_2;\cup_{k=1}^j I_k)}\\
	&{}\le C\norm{F(v)}_{\dot{W}^\s (\ol{P}_2;\cup_{k=1}^{j-1} I_k)} + C \norm{v}_{\dot{W}^{\s}(P_2;I_j)}(\norm{v}_{X(I_j)}^{k_1-1} + \norm{v}_{X(I_j)}^{k_2-1}) <\I.
\end{align*}
Fix $m$ small to obtain
\[
	\norm{v}_{\dot{W}^\sigma (P_2;\cup_{k=1}^j)}
	\le C\norm{V}_{\dot{W}^\sigma (P_2;I')} +\norm{F(v)}_{\dot{W}^\s (\ol{P}_2;\cup_{k=1}^{j-1} I_k)}+ \frac12 \norm{v}_{\dot{W}^{\s}(P_2;I_j)},
\]
showing $\norm{v}_{\dot{W}^\sigma (P_2;\cup_{k=1}^j)} < \I$. The base case $j=1$ can be proven in a similar way.
Thus, the claim is shown by induction.
\end{proof}

\section{Proof of main results}\label{sec:5}

We are now ready to show our main results.
\begin{proof}[Proof of Theorem \ref{thm1:1}] 
Let us first show that a guide flow $V(t)=U(t)v_0$ belongs to $X_{\mathrm{loc}}(\R)$.
When $n=1$, we see from the H\"older inequality, the Sobolev embedding, and Strichartz' estimates that
\begin{align*}
	\norm{U(t)v_0}_{X(I)} &= \Lebn{U(t)v_0}{P_1; I} + \Wsn{U(t)v_0}{P_2; I} \\
	&\le |I|^{\frac1{q_1}} \norm{U(t)|\nabla|^{s_1} v_0}_{L^{\I}_t(I;L^{2})}\\ 
	&\quad + |I|^{\frac1{2p(p-1)}}\norm{U(t)|\nabla|^{s_1} v_0}_{L^{4(p-1)}_t(I;L^{\frac{2(p-1)}{p-2}})} \\ 
	&\le C(|I|^{\frac1{q_1}}+|I|^{\frac1{2p(p-1)}} ) \norm{v_0}_{\dot{H}^{s_1}}
\end{align*}
for any compact interval $I \subset \R$.
In the case $n=2$, a similar argument shows
\begin{align*}
	\norm{U(t)v_0}_{X(I)} &= \Lebn{U(t)v_0}{P_1; I} + \Wsn{U(t)v_0}{P_2; I} \\
	&\le |I|^{\frac1{q_1}} \norm{U(t)|\nabla|^{s_1} v_0}_{L^{\I}_t(I;L^{2})} + \Wsn{U(t)v_0}{P_2; I}.
\end{align*}
In both cases, we obtain $U(t)v_0 \in X_{\mathrm{loc}}(\R) \cap C(\R,\dot{H}^{s_0} \cap \dot{H}^{s_1})$.

We apply Theorem \ref{thm:aeq_lwp1} with $t_0=0$ and $V(t)=U(t)v_0$.
Then, together with Theorem \ref{thm:aeq_lwp2},
there exists a unique maximal solution $v \in X_{\mathrm{loc}}(I_{\max})$ of \eqref{iggp}.
Furthermore, using Theorem \ref{thm:regularity} with $\s = s_1$ if $d=1$, $\s=s_1$ or $\s=s_2$ if $d=2$,
we obtain $v \in  X_{\mathrm{loc}}(I_{\max}) \cap C(I_{\max}, \dot{H}^{s_0} \cap \dot{H}^{s_1})$. 

Next, we prove the continuous dependence on the initial data.
Mimicking the above estimates, for any compact $J \subset I_{\max}$ we have
\[
	\norm{U(t)v_0 - U(t)\wt{v}_0}_{X(J)} \le C(|J|) \norm{v_0-\wt{v}_0}_{\dot{H}^{s_0}\cap \dot{H}^{s_1}}.
\]
When $n=2$ then
plugging this estimate to Proposition \ref{prop1:2}, 
we see that for any $\e>0$ 
there exists a neighborhood $V \subset \dot{H}^{s_0} \cap \dot{H}^{s_1}$ of $v_0$ such that for any $\wt{v}_{0} \in V$
a corresponding solution $\wt{v} \in X(J) \cap C(J, \dot{H}^{s_0})$ to \eqref{ggp} exists on $J$ and satisfies
$\norm{v-\wt{v}}_{X(J)} + \norm{v-\wt{v}}_{L^\I(J,\dot{H}^{s_2})}\le \e$.

The proof for the remaining part is similar.
We argue by an induction argument as in the proof of Proposition \ref{prop1:2}
with a persistence-of-regularity type argument in Theorem \ref{thm:regularity}.
We omit the details.
\end{proof}

\begin{proof}[Proof of Theorem \ref{thm1:2}]
It is an immediate consequence of Theorem \ref{thm:regularity} and property of $U(t)$.
\end{proof}

Here, we give a standard criterion for blowup of the solution to \eqref{ggp}. 
It will be employed in the proof of Theorem \ref{thm1:5}. 

\begin{pro}[Blowup criterion]\label{pro5:1}
Assume $n=1$, $2$ and $1+k_{\mathrm{St}} < p  <1+k_{\mathrm{m}}$. Let $v_{0} \in \dot{H}^{s_0} \cap \dot{H}^{s_1}$ and Let $v(t) \in X_{\mathrm{loc}}(I_{\max}) \cap C(I_{\max}, \dot{H}^{s_0} \cap \dot{H}^{s_1})$ be a corresponding solution given in Theorem \ref{thm1:1}.
If $T_{\max} < \infty$, then 
\[
	\norm{v}_{X([0, T])} \to \I
\]
as $T \uparrow T_{\max}$. A similar assertion holds for backward time direction.
\end{pro}

\begin{proof}[Proof of Proposition \ref{pro5:1}]
Assume that $T_{\max}<\I$ and
\[
	\lim_{T \uparrow T_{\max}}\norm{v}_{X([0,T])} < \I
\]
for contradiction.
Let us show that, under the assumption, we can extend the solution to \eqref{ggp} beyond $T_{\max}$.
By Theorem \ref{thm1:1}, we see that
there exists $t_{1} \in I_{\max}$ such that 
\begin{align}
\begin{aligned}
	&\norm{U(t-t_1)v(t_1)}_{X([t_1, T_{\max}))} \\
	\le{}& \norm{v}_{X([t_1, T_{\max}))} \\
	&+C\(\norm{v}_{X([t_1, T_{\max}))} ^{k_1-1} + \norm{v}_{X([t_1, T_{\max}))} ^{k_2-1} \)\norm{v}_{X([t_1, T_{\max}))} \\
	\le{}& \frac{\d_0}{2},
\end{aligned}
	\label{assb1:1}
\end{align}
where $\d_0$ is the constant given in Theorem \ref{thm:aeq_lwp1}. On the other hand,
as in the proof of Theorem \ref{thm1:1},
 it follows from $v(t_1) \in \dot{H}^{s_0} \cap \dot{H}^{s_1}$, Strichartz estimate and Sobolev embedding that
\begin{align}
	U(t-t_1)v(t_1) \in C(\R, \dot{H}^{s_0} \cap \dot{H}^{s_1}) \cap X_{loc}(\R). \label{assb1:2}
\end{align}
Combining \eqref{assb1:1} with \eqref{assb1:2}, there exists $\e >0$ such that 
\begin{align*}
	\norm{U(t-t_1)v(t_1)}_{X([t_1, T_{\max} + \e])} \le \d_0.
\end{align*}
By Theorem \ref{thm:aeq_lwp1}, we can construct a solution to \eqref{ggp} in the interval $[t_1, T_{\max}+\e]$. This contradicts to the definition of $T_{\max}$, which yields the desired assertion.
\end{proof}

\begin{proof}[Proof of Theorem \ref{thm1:5}]
We prove the first assertion. Suppose $\norm{v}_{X([0,T_{\max}))}<\I$.
It is immediate to see that $T_{\max}=\I$.
Indeed, we see from Proposition \ref{pro5:1} that if $T_{\max}<\I$ then
$\norm{v}_{X((0, T_{\max}))}=\I$. 
Further, by a persistence of regularity type argument, we see that $\norm{v}_{X([0,\I))} <\I$ implies
$v(t) \in C([0,\I), \dot{H}^{s_0} \cap \dot{H}^{s_1}) \cap \dot{W}^{s_0} (P_2;[0,\I)) \cap \dot{W}^{s_0} (P_2;[0,\I))$. 
Let us prove $v(t)$ scatters in $\dot{H}^{s_0} \cap \dot{H}^{s_1}$ forward in time.
Let 
$0<t_1 < t_2$. Set $s=s_0$ or $s=s_1$.
Since $U(t)$ is unitary on $\dot{H}^{s_0} \cap \dot{H}^{s_1}$, by Lemma \ref{lem1:4}, 
we have
\begin{align*}
	&{}\norm{U(-t_2)v(t_2)-U(-t_1)v(t_1)}_{\dot{H}^{s}} \\
  	\le&{}\norm{v(\cdot)-U(t_2-t_1)v(t_1)}_{L^\I((t_1,\I);\dot{H}^{s})} \\
  	=&{} \norm{\int_{t_1}^t U(t-s) F(v(s)) ds }_{L^\I((t_1,\I);\dot{H}^{s})} \\
  	\le &{}  C \norm{F(v)}_{\dot{W}^{s}(\ol{P}_2;(t_1,\I))} \\
 	\le &{} C\norm{v}_{\dot{W}^{s} (P_2;(t_1,\I))}(\norm{v}_{X((t_1,\I))}^{k_1-1} + C\norm{v}_{X((t_1,\I))}^{k_2-1})
 	\to 0
\end{align*}
as $t_1 \to\I$. This implies that $U(-t)v(t)$ converges in $\dot{H}^{s_0} \cap \dot{H}^{s_1}$ as $t\to\I$.

The second assertion is shown by an analogous argument.
By persistence of regularity, the assumption $v(t_0) \in \dot{H}^\s$ yields $v\in C([0,\I),\dot{H}^\s) \cap 
\dot{W}^{\s}(P_2; [0,\I))$. Then, arguing as above,
\begin{multline*}
\norm{U(-t_2)v(t_2)-U(-t_1)v(t_1)}_{\dot{H}^{\s}} \\
\le C \norm{v}_{\dot{W}^{\s} (P_2;(t_1,\I))}(\norm{v}_{X((t_1,\I))}^{k_1-1} + C\norm{v}_{X((t_1,\I))}^{k_2-1})
\end{multline*}
for any $0<t_1<t_2$.
\end{proof}

\begin{proof}[Proof of Theorem \ref{thm1:3}]
From Lemma \ref{stri1:1}, it holds that
\begin{align*}
\norm{U(t)v_0}_{X(\R)} &= \Lebn{U(t)v_0}{P_1} + \Wsn{U(t)v_0}{P_2} \\
&\le C(\norm{v_0}_{L^{\frac{n(p-2)}{2}}} + \norm{v_0}_{\dot{H}^{s_2}}) 
\le C\d.
\end{align*}
If we choose $\d$ such that $C\d < \d_0$, where $\delta_0$ is the 
constant given in Theorem \ref{thm:aeq_lwp1}, then we deduce from Theorem \ref{thm:aeq_lwp1} that
a corresponding solution $v$ is global. As in the proof of Theorem \ref{thm1:5}, we see that the solution scatters for both time direction in $\dot{H}^{s_2}$.
\end{proof}
\begin{proof}[Proof of Theorem \ref{thm1:4}]
By Strichartz estimate, $\norm{U(t) v_0}_{\dot{W}^{s_2}(P_2)} \le C \norm{v_0}_{\dot{H}^{s_2}}$.
On the other hand, Strichartz estimate in weighted space (see e.g. \cite{NO,Ma2}) gives us
\[
	\norm{U(t) v_0}_{L(P_1)}
	\le C \norm{|x|^{\frac2{p-2}-\frac{n}2}v_0}_{L^2}.
\]
Combining these two estimates and the assumption,
we have $\norm{U(t)v_0}_{X(\R)} \le C \delta$.
The rest is the same as in Theorem \ref{thm1:3}.
\end{proof}


\begin{bibdiv}
\begin{biblist}
\bib{MR1669387}{article}{
    author = {Bethuel, Fabrice},
    author = {Saut, Jean-Claude},
     title = {Travelling waves for the {G}ross-{P}itaevskii equation. {I}},
     date = {1999},
     issn = {0246-0211},
   journal = {Ann. Inst. H. Poincar\'e Phys. Th\'eor.},
   volume = {70},
   number = {2},
  pages = {147\ndash 238},
  url = {http://www.numdam.org/item?id=AIHPA_1999__70_2_147_0},
  review = {\MR{1669387}},      
}

\bib{MR1124294}{article}{
    author = {Christ, F.~M.},
    author = {Weinstein, M.~I.},
     title = {Dispersion of small amplitude solutions of the generalized {K}orteweg-de {V}ries equation},
     date = {1991},
     issn = {0022-1236},
   journal = {J. Funct. Anal.},
   volume = {100},
   number = {1},
  pages = {87\ndash 109},
  url = {http://dx.doi.org/10.1016/0022-1236(91)90103-C},
  review = {\MR{1124294}},      
}

\bib{MR1500278}{proceedings}{
    editor  = {Farina, Alberto},
    editor = {Saut, Jean-Claude},
    title  = {Stationary and time dependent {G}ross-{P}itaevskii equations},
     date = {2008},
     isbn  = {978-0-8218-4357-4},
  publisher = {American Mathematical Society, Providence, RI},
   volume = {473},
   series  = {Contemporary Mathematics},
  pages = {viii+180},
  url = {http://dx.doi.org/10.1090/conm/473},
  review = {\MR{1500278}},      
}

\bib{Fo}{article}{
    author = {Foschi, Damiano},
     title = {Inhomogeneous {S}trichartz estimates},
     date = {2005},
     issn = {0219-8916},
   journal = {J. Hyperbolic Differ. Equ.},
   volume = {2},
   number = {1},
  pages = {1\ndash 24},
  url = {http://dx.doi.org/10.1142/S0219891605000361},
  review = {\MR{2134950 (2006a:35043)}},      
}

\bib{PhysRevLett.69.1644}{article}{
    author = {Frisch, T.},
    author = {Pomeau, Y.},
    author = {Rica, S.},
     title = {Transition to dissipation in a model of superflow},
     date = {1992},
   journal = {Phys. Rev. Lett.},
   volume = {69},
   number = {11},
  pages = {1644\ndash 1647},
  url = {http://link.aps.org/doi/10.1103/PhysRevLett.69.1644},     
}

\bib{MR2424376}{article}{
    author = {Gallo, Cl{\'e}ment},
     title = {The {C}auchy problem for defocusing nonlinear {S}chr\"odinger equations with non-vanishing initial data at infinity},
     date = {2008},
     issn = {0360-5302},
   journal = {Comm. Partial Differential Equations},
   volume = {33},
   number = {4-6},
  pages = {729\ndash 771},
  url = {http://dx.doi.org/10.1080/03605300802031614},
  review = {\MR{2424376}},      
}

\bib{Ge}{article}{
    author = {G{\'e}rard, P.},
     title = {The {C}auchy problem for the {G}ross-{P}itaevskii equation},
     date = {2006},
     issn = {0294-1449},
   journal = {Ann. Inst. H. Poincar\'e Anal. Non Lin\'eaire},
   volume = {23},
   number = {5},
  pages = {765\ndash 779},
  url = {http://dx.doi.org/10.1016/j.anihpc.2005.09.004},
  review = {\MR{2259616}},      
}

\bib{Ge2}{incollection}{
    author = {G{\'e}rard, Patrick},
     title = {The {G}ross-{P}itaevskii equation in the energy space},
     date = {2008},
     booktitle = {Stationary and time dependent {G}ross-{P}itaevskii equations},
     issn = {},
   journal = {},
   volume = {473},
   number = {},
  pages = {129\ndash 148},
  url = {http://dx.doi.org/10.1090/conm/473/09226},
  review = {\MR{2522016}},      
}

\bib{GOV}{article}{
    author = {Ginibre, J},
    author = {Ozawa, T.},
    author = {Velo, G.},
     title = {On the existence of the wave operators for a class of nonlinear {S}chr\"odinger equations},
     date = {1994},
     issn = {0246-0211},
   journal = {Ann. Inst. H. Poincar\'e Phys. Th\'eor.},
   volume = {60},
   number = {2},
  pages = {211\ndash 239},
  url = {http://www.numdam.org/item?id=AIHPA_1994__60_2_211_0},
  review = {\MR{1270296}},      
}

\bib{10.1063/1.1703944}{article}{
    author = {Gross, Eugene~P.},
     title = {Hydrodynamics of a Superfluid Condensate},
     date = {1963},
   journal = {Journal of Mathematical Physics},
   volume = {4},
   number = {2},
  pages = {195\ndash 207},
  url = {http://scitation.aip.org/content/aip/journal/jmp/4/2/10.1063/1.1703944},
}

\bib{MR2231117}{article}{
    author = {Gustafson, Stephen},
    author = {Nakanishi, Kenji},
    author = {Tsai, Tai-Peng},
     title = {Scattering for the {G}ross-{P}itaevskii equation},
     date = {2006},
     issn = {1073-2780},
   journal = {Math. Res. Lett.},
   volume = {13},
   number = {2-3},
  pages = {273\ndash 285},
  url = {http://dx.doi.org/10.4310/MRL.2006.v13.n2.a8},
  review = {\MR{2231117}},      
}

\bib{MR2360438}{article}{
    author = {Gustafson, Stephen},
    author = {Nakanishi, Kenji},
    author = {Tsai, Tai-Peng},
     title = {Global dispersive solutions for the {G}ross-{P}itaevskii equation in two and three dimensions},
     date = {2007},
     issn = {1424-0637},
   journal = {Ann. Henri Poincar\'e},
   volume = {8},
   number = {7},
  pages = {1303\ndash 1331},
  url = {http://dx.doi.org/10.1007/s00023-007-0336-6},
  review = {\MR{2360438}},      
}

\bib{MR2559713}{article}{
    author = {Gustafson, Stephen},
    author = {Nakanishi, Kenji},
    author = {Tsai, Tai-Peng},
     title = {Scattering theory for the {G}ross-{P}itaevskii equation in three dimensions},
     date = {2009},
     issn = {0219-1997},
   journal = {Commun. Contemp. Math.},
   volume = {11},
   number = {4},
  pages = {657\ndash 707},
  url = {http://dx.doi.org/10.1142/S0219199709003491},
  review = {\MR{2559713}},      
}

\bib{MR1275405}{incollection}{
 author               = {Kato, Tosio},
 booktitle            = {Spectral and scattering theory and applications},
 pages                = {223--238},
 publisher            = {Math. Soc. Japan, Tokyo},
 series               = {Adv. Stud. Pure Math.},
 title                = {An {$L\sp {q,r}$}-theory for nonlinear {S}chr\"odinger equations},
 volume               = {23},
 date                 = {1994},
 review = {\MR{1275405}},  
 }

\bib{KPV}{article}{
    author = {Kenig, Carlos E.},
    author = {Ponce, Gustavo},
    author = {Vega, Luis},
     title = {Well-posedness and scattering results for the generalized {K}orteweg-de {V}ries equation via the contraction principle},
     date = {1993},
     issn = {0010-3640},
   journal = {Comm. Pure Appl. Math.},
   volume = {46},
   number = {4},
  pages = {527\ndash 620},
  url = {http://dx.doi.org/10.1002/cpa.3160460405},
  review = {\MR{1211741}},      
}

\bib{KMMV}{article}{
    author = {Killip, Rowan},
    author = {Masaki, Satoshi},
    author = {Jason, Murphy},
    author = {Monica, Visan},
     title = {Large data mass-subcritical {N}{L}{S}: critical weighed bound imply scattering},
     date = {2016},
   journal = {preprint},
   eprint  = {arXiv:1606.01512},     
}

\bib{MR3039823}{article}{
    author = {Killip, Rowan},
    author = {Oh, Tadahiro},
    author = {Pocovnicu, Oana},
    author = {Vi{\c{s}}an, Monica},
     title = {Global well-posedness of the {G}ross-{P}itaevskii and cubic-quintic nonlinear {S}chr\"odinger equations with non-vanishing boundary conditions},
     date = {2012},
     issn = {1073-2780},
   journal = {Math. Res. Lett.},
   volume = {19},
   number = {5},
  pages = {969\ndash 986},
  url = {http://dx.doi.org/10.4310/MRL.2012.v19.n5.a1},
  review = {\MR{3039823}},      
}

\bib{Ko}{article}{
    author = {Koh, Youngwoo},
     title = {Improved inhomogeneous {S}trichartz estimates for the {S}chr\"odinger equation},
     date = {2011},
     issn = {0022-247X},
   journal = {J. Math. Anal. Appl.},
   volume = {373},
   number = {1},
  pages = {147\ndash 160},
  url = {http://dx.doi.org/10.1016/j.jmaa.2010.06.019},
  review = {\MR{2684466 (2011j:35224)}},      
}

\bib{Ma1}{article}{
 author               = {Masaki, Satoshi},
 eprint               = {arXiv:1301.1742},
 journal              = {preprint},
 title                = {On minimal non-scattering solution to focusing mass-subcritical nonlinear {S}chr\"odinger equation},
 date                 = {2013},
 }

\bib{Ma2}{article}{
    author = {Masaki, Satoshi},
     title = {A sharp scattering condition for focusing mass-subcritical nonlinear {S}chr\"odinger equation},
     date = {2015},
     issn = {1534-0392},
   journal = {Commun. Pure Appl. Anal.},
   volume = {14},
   number = {4},
  pages = {1481\ndash 1531},
  url = {http://dx.doi.org/10.3934/cpaa.2015.14.1481},
  review = {\MR{3359531}},      
}

\bib{Ma3}{article}{
 author               = {Masaki, Satoshi},
 eprint               = {arXiv:1605.09234},
 journal              = {preprint},
 title                = {Two minimization problems on non-scattering solutions to mass-subcritical nonlinear{S}chr\"odinger equation},
 date                 = {2016},
 }

\bib{MS}{article}{
    author = {Masaki, Satoshi},
    author = {Segata, Jun-ichi},
     title = {On the well-posedness of the generalized {K}orteweg--de
              {V}ries equation in scale-critical {$\hat L{}^r$}-space},
     date = {2016},
     issn = {2157-5045},
   journal = {Anal. PDE},
   volume = {9},
   number = {3},
  pages = {699\ndash 725},
  url = {http://dx.doi.org/10.2140/apde.2016.9.699},
  review = {\MR{3518534}},      
}

\bib{MR3194504}{article}{
    author = {Miyazaki, Hayato},
     title = {The derivation of the conservation law for defocusing nonlinear {S}chr\"odinger equations with non-vanishing initial data at infinity},
     date = {2014},
     issn = {0022-247X},
   journal = {J. Math. Anal. Appl.},
   volume = {417},
   number = {2},
  pages = {580\ndash 600},
  url = {http://dx.doi.org/10.1016/j.jmaa.2014.03.055},
  review = {\MR{3194504}},      
}

\bib{MR1991146}{article}{
    author = {Nakamura, Makoto},
    author = {Ozawa, Tohru},
     title = {Small data scattering for nonlinear {S}chr\"odinger wave and {K}lein-{G}ordon equations},
     date = {2002},
     issn = {0391-173X},
   journal = {Ann. Sc. Norm. Super. Pisa Cl. Sci. (5)},
   volume = {1},
   number = {2},
  pages = {435\ndash 460},
  review = {\MR{1991146}},      
}

\bib{NO}{article}{
    author = {Nakanishi, Kenji},
    author = {Ozawa, Tohru},
     title = {Remarks on scattering for nonlinear {S}chr\"odinger equations},
     date = {2002},
     issn = {1021-9722},
   journal = {NoDEA Nonlinear Differential Equations Appl.},
   volume = {9},
   number = {1},
  pages = {45\ndash 68},
  url = {http://dx.doi.org/10.1007/s00030-002-8118-9},
  review = {\MR{1891695}},      
}

\bib{pitaevskii1961vortex}{article}{
    author = {Pitaevskii, LP},
     title = {Vortex linex in an imperfect Bose gas},
     date = {1961},
   journal = {Sov. Phys. JETP},
   volume = {13},
  pages = {451},   
}

\bib{MR1696311}{book}{
    author = {Sulem, Catherine},
    author = {Sulem, Pierre-Louis},
     title = {The nonlinear {S}chr\"odinger equation},
     date = {1999},
     isbn = {0-387-98611-1},
    note = {Self-focusing and wave collapse},
 pages  = {xvi+350},
 publisher = {Springer-Verlag, New York},
 series = {Applied Mathematical Sciences},
   volume = {139},
   number = {},
  pages = {},
  url = {},
  review = {\MR{1696311}},      
}

\bib{Vi}{article}{
    author = {Vilela, M. C.},
     title = {Inhomogeneous {S}trichartz estimates for the {S}chr\"odinger equation},
     date = {2007},
     issn = {0002-9947},
   journal = {Trans. Amer. Math. Soc.},
   volume = {359},
   number = {5},
  pages = {2123\ndash 2136 (electronic)},
  url = {http://dx.doi.org/10.1090/S0002-9947-06-04099-2},
  review = {\MR{2276614 (2008a:35226)}},      
}

\bib{MR1831831}{book}{
    author = {Zhidkov, Peter~E.},
     title = {Korteweg-de {V}ries and nonlinear {S}chr\"odinger equations: qualitative theory},
     date = {2001},
     isbn = {3-540-41833-4},
   pages = {vi+147},
  publisher  = {Springer-Verlag, Berlin},
  series = {Lecture Notes in Mathematics},
  volume  = {1756},
  review = {\MR{1831831}},      
}

\end{biblist}
\end{bibdiv}


\end{document}